\documentclass[12pt]{amsart}
\usepackage{latexsym, amsmath, amscd, amssymb, amsthm} 

\def\CT{\mathop{\mathrm{CT}}}
\newtheorem{thm}{Theorem}[section]
 \newtheorem{prop}{Proposition}[section]
\begin{document}

\noindent
\title[MacMahon Partition Analysis and the Poincar\'{e} series]{MacMahon Partition Analysis and the Poincar\'{e} series of the algebras of invariants of ternary,  quaternary and quinary forms }

\author{Leonid  Bedratyuk}
\address{ Khmelnytsky National University, Instituts'ka st. 11, Khmelnytsky, 29016, Ukraine}
\email{leonid.uk@gmail.com}

\author{Guoce Xin}
\address{Department of Mathematics, Capital Normal University, Beijing 100048, PR China}
\email{guoce.xin@gmail.com}
\begin{abstract}
By using MacMahon partition analysis technique, the Poincar\'e  series  for the algebras of  invariants of the  ternary, quaternary and quinary forms of small orders   are  calculated. 
\end{abstract}

\maketitle

\section{Introduction}
Let $V_{n,d}$ be the  vector $\mathbb{C}$-space of  $n$-ary forms of degree $d$ endowed with the natural action of the group  $SL_n.$ Consider the corresponding action of the group  (and the Lie algebra $\mathfrak{sl_{n}}$) on the algebra of polynomial functions   $\mathbb{C}[V_{n,d}]$   on the vector  space $V_{n,d}.$ Denote  by  $\mathcal{I}_{n,d}:=\mathbb{C}[V_{n,d}]^{\,SL_n}$ the  algebra of   $SL_n$-invariant polynomial functions.
In the language  of classical invariant theory   the algebra    $\mathcal{I}_{n,d}$ is called     the  algebra of invariants for   $n$-ary form of  degree $d.$ The algebra  $\mathcal{I}_{n,d}$ was a major object of research in the classical invariant theory
of the 19th century.  Nevertheless its  full description is  known only in some  particular   cases: $n=2,d\leq 10;$ $n=3,d\leq 4;$ $d=2$ for any $n.$

A possible approach to the study of the algebra $\mathcal{I}_{n,d}$ is  a description of its grading subspaces.
The algebra $\mathcal{I}_{n,d}$  is  a graded algebra
$$
\mathcal{I}_{n,d}=(\mathcal{I}_{n,d})_0 \oplus (\mathcal{I}_{n,d})_1 \oplus \cdots \oplus (\mathcal{I}_{n,d})_k \oplus  \cdots,
$$
here  $(\mathcal{I}_{n,d})_k$ is the vector subspace of   homogeneous  invariants of degree $k.$
The formal power series
$$
\mathcal{P}(\mathcal{I}_{n,d},t):=\sum_{k=0}^{\infty }\dim(\mathcal{I}_{n,d})_k t^k
$$
are  called the Poincar\'e series   of the algebra invariants $\mathcal{I}_{n,d}.$
 The Poincar\'e series $\mathcal{P}(A,t)$ of a graded algebra $A$  encodes important information about the algebra. For example,  its transcendence degree is equal to the pole order of $\mathcal{P}(A,t)$ at $t=1.$ Also, the knowledge of $\mathcal{P}(A,t)$ allows to get the upper bound for  the degree of elements  of  a  minimal generating set of the algebra $A.$ All known efficient algorithms for finding  minimal generating systems of the algebras of invariants  using the Poincar\'e series, see \cite{DerK}, \cite{Stur}.  The finitely generation of the algebra  $\mathcal{I}_{n,d}$  implies  that its  Poincar\'e series is   the power series expansions  of certain  rational function. We consider here the problem of computing efficiently this rational function. It could  be the first step towards describing  the  algebras of invariants.

 For  the case of  binary forms   $n=2,$ $d\leq 10,$ $d=12$  the  Poincar\'e series $\mathcal{P}(\mathcal{I}_{2,d},t)$   were calculated by Sylvester and  Franklin, see  \cite{SF}, \cite{Sylv-12}. They used the Cayley-Sylvester formula for the dimension of graded subspaces of the algebra invariants of binary  form. Relatively recently, Springer \cite{SP} derived  the  formula for computing the Poincar\'e  series of the algebras of invariants of the binary $d$-forms. This formula has been used by Brouwer and Cohen  \cite{BC} for  the Poincar\'e  series calculations in the cases  $d\leq 17$ and also by Littelmann and Procesi  \cite{LP} for even  $d\leq 36.$   The explicit form of the Poincar\'e series for  $d\leq 30$ is given in     \cite{BI}.
The Poincar\'e series for the algebras of invariants of ternary forms were calculated for  the case  $d=3$  in \cite{DerK} and for $d=4$   in \cite{Shi} by using of  Molien's formula and the residue theorem.  For  the case $n$-ary form $n>3$  we  do not know any results  about  the Poincar\'e series.

In the paper \cite{B_n}, the first author derived a general formula for  calculating the Poincar\'e series $\mathcal{P}(\mathcal{I}_{n,d},t).$ The formula is transformed to a constant term in the scope of MacMahon Partition Analysis,
which have been recently translated into computer software by
Andrews et al. \cite{AndrewsMPA} and by the second author, Xin \cite{XinIterate}. By using Xin's upgraded \textsf{Maple} package Ell2,
 we are able to present for the first time  the Poincar\'e series of the algebras invariants  for  the ternary form, $d=5,6;$ for the quaternary form, $d\leq 3;$ and for quinary form, $d\leq 2.$ The package also directly computes $\mathcal{P}(\mathcal{I}_{2,d},t)$ for $d\le 58$, confirming the truth of the Dixmier conjecture \cite{Dix} on the denominator in these cases.

\section{The working formulas}

Let $E_{k,i}$ denote the matrix that has a one in the $k$-th row and $i$-th column and  zeros elsewhere. Let
$$
\mathfrak{h}=\{c_1 E_{1,1} +c_2 E_{2,2}+\cdots+c_n E_{n,n} \mid c_1+c_2+\cdots+c_n=0, c_i \in \mathbb{C}\},
$$
be  the Cartan subalgebra of the Lie algebra $\mathfrak{sl_{n}}$.
Define $L_i \in \mathfrak{h}^*$ by  $L_i(E_{j,j})=\delta_{i,j}.$ Let $\beta_{i,j}=L_i-L_{j},$ ${ 1\leqslant i<j \leqslant n}$ be the positive roots of the algebra  $\mathfrak{sl_{n}}$  and let  ${\phi_i=L_1+L_2+\ldots+L_{i},}$ $ i=1,\ldots, n-1$ be  the fundamental weights.
 The matrices  $$H_1:=E_{1,\,1}{-}E_{2,\, 2}, H_2:=E_{2,\,2}{-}E_{3,\,3}, \ldots H_{n-1}:=E_{n-1,\,n-1}{-}E_{n,\,n}$$  generate the Cartan  subalgebra $\mathfrak{h}.$
It is easy to check that  $\phi_i(H_j)=\delta_{i,j}.$
 Denote by  $\lambda=(\lambda_1,\lambda_2,\ldots,\lambda_{n-1})$  the weight
$$
\lambda_1 \phi_1+\lambda_2 \phi_2+\ldots+\lambda_{n-1} \phi_{n-1}, \lambda_i \in \mathbb{Z}.
$$
In the notation  half the  sum of all positive roots  $\rho $ equals   $(1,1,\ldots,1).$ Note that $\lambda(H_i)=\lambda_i.$

The following statement holds:
\begin{thm}[\cite{B_n}] \label{t-1}
The Poincar\'e series $\mathcal{P}_{n,d}(t)$ of the algebra $\mathcal{I}_{n,d}$ equals
\begin{gather}\label{main_n}
\mathcal{P}(\mathcal{I}_{n,d},t)= \frac{1}{(2 \pi i)^{n-1}} \oint_{| \textit{\textbf{q}}|=1}   \frac{\displaystyle  \sum_{s \in \mathcal{W}} (-1)^{|s|} \textit{\textbf{q}}^{n \{\rho-s(\rho)\}'}  }{\displaystyle \prod_{|\eta| \leq  d } \left(1-t \textit{\textbf{q}}^{n \eta-d \rho}\right)} \frac{d \textit{\textbf{q}}}{ \textit{\textbf{q}}},
\end{gather}
where $\rho=(1,\ldots,1)$ is  half the sum of the positive roots, $\eta:=(\eta_1,\eta_2,\ldots,\eta_{n-1}) \in \mathbb{N}^{n-1},$ $|\eta|:=\sum \eta_i, $  $\mathcal{W}$  is the Weyl group of $\mathfrak{sl}_n,$ $\{\mu \}$ is the unique dominant weight on the orbit $\mathcal{W}(\mu)$ for  the weight ${\mu=(\mu_1,\ldots,\mu_{n-1})}$ and $\mu'=(\mu_1',\ldots,\mu_{n-1}')$
\begin{equation}\label{mu}
\displaystyle \mu_i'=\left(\sum_{s=i}^{n-2} \mu_s-\frac{1}{n}\Bigl(\,\sum_{s=1}^{n-2} s \mu_s-\mu_{n-1}\Bigr) \right), i=1,\ldots, n-1.\\
\end{equation}
\end{thm}
Here we used the multi-index notation $\textit{\textbf{q}}^{\mu}:=q_1 ^{\mu_1} \cdots q_{n-1} ^{\mu_{n-1}}$  and
$$
\oint_{|q_{n-1}|=1} \ldots  \oint_{|q_{1}|=1} f(t,q_1,q_2,\ldots,q_{n-1})  \frac{dq_1 \ldots dq_{n-1}}{q_1\ldots q_{n-1}}:=  \oint_{| \textit{\textbf{q}}|=1} f(t,\textit{\textbf{q}})  \frac{d \textit{\textbf{q}}}{ \textit{\textbf{q}}},
$$
for arbitrary  rational function $f(t,q_1,q_2,\ldots,q_{n-1}).$

 Denote by $\mathcal{B}_n(\textit{\textbf{q}})$ the numerator of the integrand  of the integral   $(\ref{main_n}).$
 It was shown in \cite{B_n} that
 $$
 \displaystyle \mathcal{B}_3(\textit{\textbf{q}})=1+q_2^3\,q_1^3+\frac{q_1^6}{q_2^3}-2\,q_1^3-q_1^6.
 $$
Let us calculate now  $\mathcal{B}_4(\textit{\textbf{q}})$  and  $\mathcal{B}_5(\textit{\textbf{q}}).$

\begin{thm} The following statement holds

\begin{gather*}
\displaystyle  \mathcal{B}_4(\textit{\textbf{q}}):=1+{\frac {{q_{{1}}}^{4}{q_{{2}}}^{4}}{{q_{{3}}}^{4}}}+{\frac {{q_{{1}}
}^{12}{q_{{2}}}^{4}}{{q_{{3}}}^{4}}}+2\,{q_{{1}}}^{8}{q_{{2}}}^{4}+2\,
{q_{{1}}}^{4}{q_{{2}}}^{4}+2\,{\frac {{q_{{1}}}^{12}}{{q_{{3}}}^{4}}}+
2\,{\frac {{q_{{1}}}^{8}}{{q_{{3}}}^{4}}}+{\frac {{q_{{1}}}^{8}{q_{{2}
}}^{8}}{{q_{{3}}}^{8}}}-3\,{q_{{1}}}^{4}-\\-{q_{{1}}}^{12}-2\,{\frac {{q_
{{1}}}^{8}{q_{{2}}}^{4}}{{q_{{3}}}^{4}}}-{\frac {{q_{{1}}}^{8}{q_{{2}}
}^{8}}{{q_{{3}}}^{4}}}-{\frac {{q_{{1}}}^{12}{q_{{2}}}^{4}}{{q_{{3}}}^
{8}}}-{q_{{1}}}^{4}{q_{{2}}}^{4}{q_{{3}}}^{4}-2\,{q_{{1}}}^{8}-{\frac
{{q_{{1}}}^{12}}{{q_{{2}}}^{4}{q_{{3}}}^{4}}}.
\end{gather*}

\end{thm}
\begin{proof}
For  the case  of ternary form we have $\phi_1=L_1, \phi_2=L_1+L_2,$ $\phi_3=L_1+L_2+L_3.$ Then  half the sum of the positive roots $\rho$ in the fundamental weight basis is equal to  $(1,1,1).$

Consider the four  positive roots
$$
\begin{array}{l}
\beta_1:=L_1-L_2=2\phi_1-\phi_2=(2,-1,0),\\
\beta_2:=L_2-L_3=-\phi_1+2\phi_2=(-1,2,0),\\
\beta_3:=L_3-L_4=2\phi_3-\phi_2=(0,-1,2),\\
\beta_4:=L_1-L_4=\phi_1+\phi_3=(1,0,1).
\end{array}
$$
 The Weyl group of Lie  algebra  $\mathfrak{sl_{4}}$  is  generated by the  four  reflections  $s_{\beta_1},$ $s_{\beta_2},$ $s_{\beta_3},$ $s_{\beta_4}$ and  consists of 24 elements. The reflections $s_{\beta_i}$ act on a weight $\lambda=(\lambda_1,\lambda_2,\lambda_3)$ by  $s_{\beta_i}(\lambda)=\lambda - \lambda_i \beta_i,$ $i=1,2,3$ and by  $s_{\beta_4}(\lambda)=\lambda-(\lambda_1+\lambda_2+\lambda_3) \beta_4.$ It is easy to see that the stabilizer $\mathcal{W}_{\rho}$ is the trivial subgroup. Then  the orbit-stabilizer theorem implies that  the orbit  $\mathcal{W}(\rho)$ consists of  $24$ weights. Let us divide the element of   $\mathcal{W}(\rho)$ into 2 parts. Denote by  $\mathcal{W}(\rho)^{+},\mathcal{W}(\rho)^{-}$ the subsets  whose elements are the weights  generated by even and odd reflections. We  have
 $$
\begin{array}{c}
    \mathcal{W}(\rho)^{+}= \{(-1,-1,-1),(1,1,1),(1,-2,3),(-1,3,-1),(-3,2,-1),(3,-2,1),\\ (-1,2,-3),(-2,1,2),(2,1,-2),(2,-1,-2),(-2,-1,2),(1,-3,1)\}, \\
\end{array}
$$
and
 $$
\begin{array}{c}
    \mathcal{W}(\rho)^{-}= \{(-1,2,1),(1,-2,-1),(1,2,-1),(2,-1,2),(-2,1,-2),(-1,-2,1),\\ (-1,-1,3),(1,1,-3),(2,-3,2),(-2,3,-2),(3,-1,-1),(-3,1,1)\}. \\
\end{array}
$$

Therefore
 $$
\begin{array}{c}
\rho-\mathcal{W}(\rho)^{+}=\{(-1,2,3),(4,-1,2),(2,2,2),(0,4,0),(0,0,0),\\
(0,3,-2),(-1,0,3),(2,-2,2),(3,2,-1),(-2,3,0),(3,0,-1),(2,-1,4) \},
\end{array}
$$
 $$
\begin{array}{c}
\rho-\mathcal{W}(\rho)^{-}=\{ (2,3,0),(2,2,-2),(-2,2,2),(0,0,4),(-1,4,-1),\\
(3,-2,3),(2,-1,0),(-1,2,-1),(0,-1,2),(3,0,3),(0,3,2),(4,0,0) \}.
\end{array}
$$
Recall that $\{\mu\}$  denotes the unique dominant weight on the orbit $\mathcal{W}(\mu).$  Then for the elements of $\rho-\mathcal{W}(\rho)^{+}$  we  have
$$
\begin{array}{ll}
\{(-1,2,3)\}=(1,1,3),\{(4,-1,2)\}=(3,1,1),\{(2,2,2)\}=(2,2,2),\{(0,4,0)\}=(0,4,0),\\ \{(0,0,0)\}=(0,0,0), \{(0,3,-2)\}=\{(-1,0,3)\}=(0, 1, 2),\{(2,-2,2)\}=(0, 2, 0),\\
\{(3,2,-1)\}=(3, 1, 1),\{(-2,3,0)\}=\{(3,0,-1)\}=(2, 1, 0),\{(2,-1,4)\}=(1, 1, 3).
\end{array}
$$
Similarly for the set $\rho-\mathcal{W}(\rho)^{-}$  we get
$$
\begin{array}{ll}
\{(2,3,0)\}=(2,3,0),\{(2,2,-2)\}=(2, 0, 2),\{(-2,2,2)\}=(2, 0, 2),\{(0,0,4)\}=(0,4,0),\\ \{(-1,4,-1)\}=(1, 2, 1), \{(3,-2,3)\}=(1, 2, 1),\{(2,-1,0)\}=\{(-1,2,-1)\}=(1, 0, 1),\\
\{(0,-1,2)\}=(1, 0, 1),\{(3,0,3)\}=(3,0,3),\{(0,3,2)\}=(0,3,2),\{(4,0,0)\}=(4,0,0).
\end{array}
$$

By $(\ref{mu})$ we get
$$
(\mu_1,\mu_2,\mu_3)'=\left( \frac{3\,\mu_{{1}}+2\,\mu_{{2}}+\mu_{{3}}}{4} ,\frac{-\mu_{{1}}-2\,\mu_{{2}}+\,\mu_{{3}}}{4},\frac{-\,\mu_{{1}}+2\,\mu_{{2}}+\mu_
{{3}}}{4}  \right).
$$
It implies that for  the set $\{\rho-\mathcal{W}(\rho)^{+}\}'$ we have
$$
\begin{array}{c}
(1,1,3)'=(2, 1, 0),(3,1,1)'=(3, 0, -1),(2,2,2)'=(3, 1, -1),(0,4,0)'=(2, 2, -2),\\
(0,0,0)'=(0,0,0), (0, 1, 2)'=(1, 1, 0),(0, 2, 0)'=(1, 1, -1),\\
(3, 1, 1)'=(3, 0, -1),(2, 1, 0)'=(2, 0, -1),(1, 1, 3)'=(2, 1, 0)\end{array}
$$
and for $\{\rho-\mathcal{W}(\rho)^{-}\}'$
$$
\begin{array}{ll}
(2,3,0)'=(3, 1, -2),(2, 0, 2)'=(2, 0, 0),(2, 0, 2)'=(2, 0, 0),(0,4,0)'=(2, 2, -2),\\
 (1, 2, 1)'=(2, 1, -1), (1, 2, 1)'=(2, 1, -1),(1, 0, 1)'=(1, 0, 0),\\
(3,0,3)'=(3, 0, 0),(0,3,2)'=(2, 2, -1),(4,0,0)'=(3, -1, -1).
\end{array}
$$
Now we may calculate  the numerator  of the integrand (\ref{main_n})
\begin{gather*}
\mathcal{B}_4(q_1,q_2,q_3)=\sum_{s \in \mathcal{W}} (-1)^{|s|} \textit{\textbf{q}}^{4 \{\rho-s(\rho)\}'}=\sum_{\mu \in \{\rho-\mathcal{W}(\rho)^{+}\}'}\textit{\textbf{q}}^{4 \mu}-\sum_{\mu \in \{\rho-\mathcal{W}(\rho)^{-}\}'}\textit{\textbf{q}}^{4 \mu}=\\ \\
=1+{\dfrac {{q_{{1}}}^{4}{q_{{2}}}^{4}}{{q_{{3}}}^{4}}}+{\frac {{q_{{1}}
}^{12}{q_{{2}}}^{4}}{{q_{{3}}}^{4}}}+2\,{q_{{1}}}^{8}{q_{{2}}}^{4}+2\,
{q_{{1}}}^{4}{q_{{2}}}^{4}+2\,{\frac {{q_{{1}}}^{12}}{{q_{{3}}}^{4}}}+
2\,{\frac {{q_{{1}}}^{8}}{{q_{{3}}}^{4}}}+{\frac {{q_{{1}}}^{8}{q_{{2}
}}^{8}}{{q_{{3}}}^{8}}}-3\,{q_{{1}}}^{4}-\\ \\ -{q_{{1}}}^{12}-2\,{\dfrac {{q_
{{1}}}^{8}{q_{{2}}}^{4}}{{q_{{3}}}^{4}}}-{\frac {{q_{{1}}}^{8}{q_{{2}}
}^{8}}{{q_{{3}}}^{4}}}-{\frac {{q_{{1}}}^{12}{q_{{2}}}^{4}}{{q_{{3}}}^
{8}}}-{q_{{1}}}^{4}{q_{{2}}}^{4}{q_{{3}}}^{4}-2\,{q_{{1}}}^{8}-{\frac
{{q_{{1}}}^{12}}{{q_{{2}}}^{4}{q_{{3}}}^{4}}}.
\end{gather*}
\end{proof}
In the similar way we obtain
\begin{gather*}
\mathcal{B}_5(q_1,q_2,q_3,q_4)=\\ 1 - {\displaystyle \frac {2\,{q_{1}}^{15}}{{q_{3}}^{5}\,{q_{4}}^{
5}}}  - {\displaystyle \frac {2\,{q_{1}}^{10}\,{q_{2}}^{5}}{{q_{4
}}^{5}}}  + {\displaystyle \frac {{q_{1}}^{20}}{{q_{4}}^{10}}}
 - {\displaystyle \frac {4\,{q_{1}}^{15}\,{q_{2}}^{5}}{{q_{4}}^{
10}}}  - {\displaystyle \frac {2\,{q_{1}}^{15}\,{q_{2}}^{10}}{{q
_{4}}^{5}}}  + {\displaystyle \frac {2\,{q_{1}}^{10}\,{q_{2}}^{10
}}{{q_{4}}^{10}}}  - {\displaystyle \frac {{q_{1}}^{20}\,{q_{2}}
^{10}}{{q_{4}}^{15}}}  \\
\mbox{} + {\displaystyle \frac {{q_{1}}^{15}\,{q_{2}}^{15}}{{q_{4
}}^{15}}}  + {\displaystyle \frac {3\,{q_{1}}^{5}\,{q_{2}}^{5}}{{
q_{4}}^{5}}}  - {\displaystyle \frac {2\,{q_{1}}^{20}\,{q_{2}}^{5
}}{{q_{4}}^{10}}}  - {\displaystyle \frac {{q_{1}}^{15}\,{q_{2}}
^{15}}{{q_{4}}^{10}}}  - {\displaystyle \frac {3\,{q_{1}}^{20}}{{
q_{3}}^{5}\,{q_{4}}^{5}}}  - 2\,{q_{1}}^{5}\,{q_{2}}^{5}\,{q_{3}}
^{5} + {\displaystyle \frac {{q_{1}}^{20}\,{q_{2}}^{10}}{{q_{4}}
^{10}}}  \\
\mbox{} + {\displaystyle \frac {2\,{q_{1}}^{15}\,{q_{2}}^{10}}{{q
_{4}}^{10}}}  + {\displaystyle \frac {2\,{q_{1}}^{20}\,{q_{2}}^{5
}}{{q_{4}}^{5}}}  - 3\,{q_{1}}^{10}\,{q_{2}}^{5}\,{q_{3}}^{5} -
{\displaystyle \frac {4\,{q_{1}}^{10}\,{q_{2}}^{10}}{{q_{4}}^{5}}
}  + {\displaystyle \frac {2\,{q_{1}}^{10}\,{q_{2}}^{10}\,{q_{3}}
^{5}}{{q_{4}}^{5}}}  \\
\mbox{} + {\displaystyle \frac {2\,{q_{1}}^{15}\,{q_{2}}^{10}}{{q
_{3}}^{5}\,{q_{4}}^{5}}}  - {\displaystyle \frac {2\,{q_{1}}^{15}
\,{q_{2}}^{15}}{{q_{3}}^{5}\,{q_{4}}^{10}}}  + {\displaystyle
\frac {2\,{q_{1}}^{10}\,{q_{2}}^{5}\,{q_{3}}^{5}}{{q_{4}}^{5}}}
 + {\displaystyle \frac {{q_{1}}^{15}\,{q_{2}}^{15}}{{q_{3}}^{5}
\,{q_{4}}^{5}}}  + {\displaystyle \frac {{q_{1}}^{15}\,{q_{2}}^{
15}}{{q_{3}}^{10}\,{q_{4}}^{10}}}  - {\displaystyle \frac {2\,{q
_{1}}^{20}\,{q_{2}}^{5}}{{q_{3}}^{5}\,{q_{4}}^{5}}}  \\
\mbox{} - {\displaystyle \frac {{q_{1}}^{20}\,{q_{2}}^{5}\,{q_{3}
}^{5}}{{q_{4}}^{10}}}  - {\displaystyle \frac {2\,{q_{1}}^{10}\,{
q_{2}}^{5}}{{q_{3}}^{5}\,{q_{4}}^{5}}}  - {\displaystyle \frac {{
q_{1}}^{10}\,{q_{2}}^{10}\,{q_{3}}^{10}}{{q_{4}}^{10}}}  +
{\displaystyle \frac {2\,{q_{1}}^{15}\,{q_{2}}^{5}}{{q_{3}}^{5}\,
{q_{4}}^{5}}}  - {\displaystyle \frac {2\,{q_{1}}^{15}\,{q_{2}}^{
10}}{{q_{3}}^{5}\,{q_{4}}^{10}}}  + {\displaystyle \frac {{q_{1}}
^{10}\,{q_{2}}^{10}}{{q_{3}}^{5}\,{q_{4}}^{5}}}  \\
\mbox{} + {\displaystyle \frac {2\,{q_{1}}^{20}\,{q_{2}}^{10}}{{q
_{3}}^{5}\,{q_{4}}^{10}}}  + {\displaystyle \frac {{q_{1}}^{20}\,
{q_{2}}^{5}\,{q_{3}}^{5}}{{q_{4}}^{15}}}  - {\displaystyle
\frac {2\,{q_{1}}^{10}\,{q_{2}}^{10}\,{q_{3}}^{5}}{{q_{4}}^{10}}
}  - {\displaystyle \frac {2\,{q_{1}}^{15}\,{q_{2}}^{5}\,{q_{3}}
^{5}}{{q_{4}}^{5}}}  + {\displaystyle \frac {2\,{q_{1}}^{20}\,{q
_{2}}^{5}}{{q_{3}}^{5}\,{q_{4}}^{10}}}  \\
\mbox{} - {\displaystyle \frac {{q_{1}}^{20}\,{q_{2}}^{10}}{{q_{3
}}^{5}\,{q_{4}}^{5}}}  + {\displaystyle \frac {2\,{q_{1}}^{15}\,{
q_{2}}^{10}\,{q_{3}}^{5}}{{q_{4}}^{10}}}  + {q_{1}}^{10}\,{q_{2}}
^{10} + 3\,{q_{1}}^{5}\,{q_{2}}^{5} + 2\,{q_{1}}^{15}\,{q_{2}}^{5
} + 4\,{q_{1}}^{10}\,{q_{2}}^{5} + {\displaystyle \frac {4\,{q_{1
}}^{15}}{{q_{4}}^{5}}}  \\
\mbox{} + {\displaystyle \frac {2\,{q_{1}}^{20}}{{q_{4}}^{5}}}
 - {\displaystyle \frac {2\,{q_{1}}^{5}\,{q_{2}}^{5}\,{q_{3}}^{5}
}{{q_{4}}^{5}}}  + {\displaystyle \frac {{q_{1}}^{20}}{{q_{2}}^{5
}\,{q_{3}}^{5}\,{q_{4}}^{5}}}  - {\displaystyle \frac {{q_{1}}^{
20}\,{q_{2}}^{10}}{{q_{3}}^{10}\,{q_{4}}^{10}}}  + {q_{1}}^{5}\,{
q_{2}}^{5}\,{q_{3}}^{5}\,{q_{4}}^{5} + {\displaystyle \frac {{q_{
1}}^{10}\,{q_{2}}^{5}\,{q_{3}}^{5}}{{q_{4}}^{10}}}  \\
\mbox{} + {\displaystyle \frac {{q_{1}}^{10}\,{q_{2}}^{10}\,{q_{3
}}^{10}}{{q_{4}}^{15}}}  + {\displaystyle \frac {2\,{q_{1}}^{15}
\,{q_{2}}^{5}\,{q_{3}}^{5}}{{q_{4}}^{10}}}  - {\displaystyle
\frac {2\,{q_{1}}^{15}\,{q_{2}}^{10}\,{q_{3}}^{5}}{{q_{4}}^{15}}
}  + {\displaystyle \frac {3\,{q_{1}}^{10}}{{q_{4}}^{5}}}  - 4\,{
q_{1}}^{5} - 2\,{q_{1}}^{15} - 3\,{q_{1}}^{10} \\
\mbox{} - {q_{1}}^{20}
\end{gather*}

\section{Computing the Poincar\'e series by MacMahon partition analysis}
Our first step for computing the Poincar\'e series is to get a constant term expression by the following well-known transformation. We present here the version best fit our situation.

  \begin{prop} \label{p-1.1} If   $f(t,q_1,q_2,\ldots ,q_n)$ is
a power series in $t$ with coefficients Laurent polynomials in $\mathbb{C}[q_1,q_2,\ldots ,q_n; 1/q_1,1/q_2,\ldots
,1/q_n]$, then
\begin{align}
 \left({1\over  2\pi}\right)^n\int_{|q_n|=1} \cdots
\int_{|q_1|=1} f(t,q_1,\dots,q_n)\frac{dq_1dq_2\cdots dq_n}{q_1\cdots q_n} = \CT_q f(t,q_1,\dots,q_n), 
\end{align}
where $\CT_q f$ means to take constant term, i.e., the coefficient of $q_1^0q_2^0\cdots q_n^0$, in $f$.
\end{prop}
\begin{proof}
By linearity, it is sufficient to show the proposition holds for each coefficient of $f$ in $t$, which is a Laurent polynomial $L(q_1,\dots,q_n)$. By linearity, again, it suffices to assume $L(q_1,q_2,\dots,q_n)=q_1^{k_1}\cdots q_n^{k_n}$ is a monomial. The proposition then follows from the fact that
  $$ \frac{1}{2\pi} \int_{|q_1|=1} q_1^{k_1} \frac{dq_1}{q_1} =\delta_{k_1,0},$$
  where $\delta_{a,b}$ is $1$ if $a=b$ and $0$ otherwise.
\end{proof}

Applying Proposition \ref{p-1.1} to Theorem \ref{t-1} gives the following result, which is the starting point of our calculation.
\begin{thm} The following formula  holds
$$
\mathcal{P}(\mathcal{I}_{n,d},t)=\CT_{q}\frac{\mathcal{B}_n(\textit{\textbf{q}})  }{\displaystyle \prod_{|\eta| \leq  d } \left(1-t \textit{\textbf{q}}^{n \eta-d \rho}\right)}.
$$
\end{thm}

The constant term in the theorem can the theoretically evaluated by the theory of MacMahon partition analysis. 
MacMahon \cite{mac} developed his Omega calculus to study partition related problems. One of his two Omega operator acting on the $\lambda$ variables is defined by
 $$ \underset{\scriptscriptstyle = }{\Omega} \sum_{i_1=-\infty}^{\infty}\cdots \sum_{i_n=-\infty}^{\infty} a_{i_1\ldots i_n}\lambda_1^{i_1}
\cdots \lambda_n^{i_n}  =a_{0,0,\dots,0},$$
 where the $a$'s are free of $\lambda$ and the summation satisfies certain convergence condition. This operator is just our constant term operator. 

It turns out that MacMahon partition analysis has wide applications, as illustrated by Andrews et al in a series of papers starting with \cite{AndrewsMPA}. Computer software such as \textsf{Mathematica} package Omega of Andrews et al \cite{Omega}, and \textsf{Maple} package Ell of Xin \cite{XinIterate} are developed.
Our calculation uses Xin's Ell2 package, which upgraded the Ell package in aim of solving the Hdd5 related problems in \cite{GarsiaXin}. The package can be downloaded from the web
{\it
http://www.combinatorics.net.cn/homepage/xin/maple/ell2.rar}. Using this package we obtain the following results.

\begin{thm} The Poincar\'e series  $\mathcal{P}(\mathcal{I}_{3,d},t)$  for  the algebra  of  invariants of ternary form of orders $1,2,3,4,5,6$ is given by
\begin{gather*}
\mathcal{P}(\mathcal{I}_{3,1},t)=1,
\mathcal{P}(\mathcal{I}_{3,2},t)=\frac{1}{1-t^3},
\mathcal{P}(\mathcal{I}_{3,3},3)={\frac {1}{ \left( 1-{t}^{4} \right)  \left( 1-{t}^{6} \right) }},\\
\mathcal{P}(\mathcal{I}_{3,4},3)=\frac{b_{3,4}(t)} {\left({1-{t}^{3}}\right)\left({1-{t}^{6}}\right)\left({1-{t}^{9}}\right)\left({1-{t}^{12}}\right)\left({1-{t}^{15}}\right)\left({1-{t}^{18}}\right)\left({1-{t}^{27}}\right)}
,\\
 b_{3,4}(t)=1+{t}^{9}+{t}^{12}+{t}^{15}+2{t}^{18}+3{t}^{21}+2{t}^{24}+3{t}^{27}+4{t}^{30}+3{t}^{33}+4{t}^{36}+4{t}^{39}\\
+3{t}^{42}+4{t}^{45}+3{t}^{48}+2{t}^{51}+3{t}^{54}+2{t}^{57}+{t}^{60}+{t}^{63}+{t}^{66}+{t}^{75}
\end{gather*}
\begin{multline*}
 \mathcal{P}(\mathcal{I}_{3,5},3)= \frac{b_{3,5}}{ \left({1-{t}^{6}}\right)\left({1-{t}^{9}}\right)\left({1-{t}^{12}}\right)\left({1-{t}^{15}}\right)
  \left({1-{t}^{18}}\right)^{2}\left({1-{t}^{21}}\right)\left({1-{t}^{24}}\right)\left({1-{t}^{27}}\right)}\\
  \times\frac1{\left({1-{t}^{30}}\right)\left({1-{t}^{33}}\right)\left({1-{t}^{36}}\right)\left({1-{t}^{48}}\right)}
 \end{multline*}
 \begin{multline*}
  \mathcal{P}(\mathcal{I}_{3,6},3)= \frac{b_{3,6}}{\left({1-{t}^{3}}\right)\left({1-{t}^{4}}\right)
\left({1-{t}^{6}}\right)^{2}\left({1-{t}^{7}}\right)^{2}
\left({1-{t}^{8}}\right)\left({1-{t}^{9}}\right)^{2}\left({1-{t}^{10}}
\right)\left({1-{t}^{11}}\right)}\\
\times \frac1{\left({1-{t}^{12}}\right)\left({1-{t}^{13}}\right)\left({1-{t}^{15}}\right)^{2}\left({1-{t}^{16}}\right)\left({1-{t}^{17}}\right)\left({1-{t}^{19}}\right)
\left({1-{t}^{20}}\right)\left({1-{t}^{25}}\right)},
\end{multline*}
where $b_{3,5}$ and $b_{3,6}$ are very huge so we give at the end of this section.
\end{thm}

\begin{thm} The Poincar\'e series  $\mathcal{P}(\mathcal{I}_{4,d},t)$  for  the algebra  of  invariants of quaternary form of orders $1,2,3$ is given by
\begin{gather*}
\mathcal{P}(\mathcal{I}_{4,1},t)=1,\qquad
\mathcal{P}(\mathcal{I}_{4,2},t)=\frac{1}{1-t^4},\\
\mathcal{P}(\mathcal{I}_{4,3},3)={\frac {1+{t}^{100}}{ \left( 1-{t}^{8} \right)  \left( 1-{t}^{16}
 \right)  \left( 1-{t}^{24} \right)  \left( 1-{t}^{32} \right)
 \left( 1-{t}^{40} \right) }},
\end{gather*}
\end{thm}
Also we calculated   the Poincar\'e series  $\mathcal{P}(\mathcal{I}_{5,d},t)$  for  the algebras of  invariants of quinary form and have got the following results:
\begin{gather*}
\mathcal{P}(\mathcal{I}_{5,1},t)=1,\qquad
\mathcal{P}(\mathcal{I}_{5,2},t)=\frac{1}{1-t^5}.%
\end{gather*}

The command we use in the package Ell2 is $E\_OeqW(Q,vxa,va)$ that computes the constant term of $Q$ in variables of $va$, in which the input $Q$ is an
\emph{Elliott-rational function} written in the form
$$ \frac{L}{(1-M_1)(1-M_2)\cdots (1-M_k)},$$
where $L$ is a Laurent polynomial and $M_i$ are monomials; $vxa$ is a list of variables that defines a field of iterated Laurent series clarifying the series expansion of each $(1-M_i)^{-1}$; and $va$ is a list of variables that are to be eliminated. The output is a (big) sum of Elliott-rational functions. One can then combine them to a single rational function if needed. It is worth noting that software such as Latte \cite{Latte} are developed for $L$ being a monomial, in which case the constant term is related to certain counting in a rational convex polytope.

To calculate $\mathcal{P}(\mathcal{I}_{n,d},t)$, the input $Q$ is clear, the $vxa$ is taken to be $[t,q_1,q_2,\dots,q_n]$, and $va$ is clearly $[q_1,\dots,q_n]$. Let us start with the $n=2$ case by the well-known formula
\begin{align}
  \mathcal{P}(\mathcal{I}_{2,d},t)= \CT_q \frac{1-q^2}{\prod_{0\le j\le d } (1-q^{2j-d})}.
\end{align}
We explain in detail how we compute the $d=3$ case. The problem is to compute the constant term of 
$$Q=\frac{ 1-{q}^{2}}{
 \left( 1-{\frac {t}{{q}^{3}}} \right) \left( 1-{\frac {t}{q}}
 \right) \left( 1-qt \right) \left( 1-{q}^{3}t \right)} .$$
Letting $vxa=[t,q]$ means that we are working in the field of iterated Laurent series defined by $0<t<\!\!<q<\!\!<1$ (where $<\!\!<$ means ``much smaller"), so that $M=q^{2j-d}t$ is small for all $j$ and we have the correct series expansion:
$$ \frac{1}{1-M} = \sum_{k\ge 0} M^k =\sum_{k\ge 0} (q^{2j-d}t)^k.$$ 
Note that if $M $ were large then we should have the alternative series expansion:
$$ \frac{1}{1-M} = \frac{1}{-M(1-1/M)}=-\sum_{k\ge 0} M^{-k-1}.$$

To extract the constant term of $Q$, we first compute the partial fraction decomposition with respect to $q$:
$$ Q= \frac{p_{-3}(q)}{q^3-t}+\frac{p_{-1}}{q-t}+\frac{p_{1}}{1-qt}+\frac{p_{3}}{1-q^3t},$$
where $p_{-3}(q)$ and $p_3(q)$ are polynomials of degree less than $3$, and $p_{-1}$ and $p_1$ are constant, with respect to $q$. Now one can check that taking constant term gives 
$\CT_q Q= p_1+p_3(0).$ To be precise, we have 
$$ p_1=\frac{1}{(1-t^2)(1-t^4)}, p_3=-\frac{t^2+qt+q^2}{(1-t^2)(1-t^4)}, \CT_q Q=p_1+p_3(0)=\frac{1}{1-q^4}.$$
From this example, we see that only part of the partial fraction decomposition is need. Indeed, this happens most of the time. The Ell package adopted an effective way to compute only the necessary partial fraction, and then extract the constant term.

For the binary case, the package computes $\mathcal{P}(\mathcal{I}_{2,d},t)$ quickly. The computation time spend on our personal laptop for some $d$  is listed in the following table, where the second row is measured by $100$ seconds.

\medskip
\begin{tabular}{|l|r|r|r|r|r|r|r|r||r|r|r|r|r|r|r|}
  \hline
  $d$ &  48 & 50 & 52 & 54 & 56 & 58 & 60 & 62   &      45 & 47 & 49 & 51 & 53 &55 &57\\
  $10^2$s  & 5.5 & 7.1 & 9.5 & 12 & 14 & 19 & 21 & 28 &  17 & 24 & 31 & 37 &52 & 62 & 76 \\
  \hline
\end{tabular}

\medskip
For example, the computation for $d=62$ spends about 2800 seconds. From the table we can see that  i) the time increased from $d$ to $d+2$ is not fast, so computation for larger $d$ should be possible; ii) odd $d$ case is much harder than even $d$ case, in fact our computation takes the constant term of $Q\Big|_{q= q^{1/2}}$ instead since $Q$ is an even function in $q$.

When $n$ becomes larger, the computational complexity increases very fast. See \cite{XinIterate}. Our computer quickly delivers $\mathcal{P}(\mathcal{I}_{3,d},t)$ for $d=1,2,3,4$, $\mathcal{P}(\mathcal{I}_{4,d},t)$ for $d=1,2$, and $\mathcal{P}(\mathcal{I}_{5,1},t)$. Only the case $(n,d)=(3,4)$ takes about 30 seconds. The cases $(n,d)=(3,5),(3,6),(4,3),(5,2)$ seems to take too much time or too much memory, but the flexibility of the package allows us to calculate these cases by the following two simple tricks.

\begin{enumerate}
  \item We can split $Q$ as $Q=Q_1+\cdots +Q_k$ using partial fraction decompositions to part of the rational function and then evaluate the constant term of each $Q_i$. For example, if $Q$ can be written as
      $$Q= Q'\cdot \frac{q_1^{(k-1)r}}{(1-q_1^rM_1)\cdots (1-q_1^r M_k)},$$
      where $M_i$ are independent of $q_1$, then the simple partial fraction decomposition of $Q/Q'$ will give rise to $Q=Q_1+\cdots+Q_k$ with each $Q_i$ simple.

  \item It is clear that for any positive integer $r$
  $$ \CT_{q} Q(t;q_1,q_2,\dots,q_n) = \CT_{q} Q(t;q_1^r,q_2^r,\dots,q_n^r),$$
  but computing the left constant term clearly save memories.
\end{enumerate}

We only need to apply trick (1) at the beginning, and might be able to apply trick (2) thereafter.
The readers are welcome to try to compute these cases by themselves. We only report here that the $(n,d)=(3,5)$ case took our computer about $6$ hours, but the $(n,d)=(3,6)$ case took only about 4 hours because trick 2 applies at the beginning.

We conclude this section by giving the following data as we promised.
\begin{align*}
  &b_{3,5}=1+{t}^{6}+16{t}^{12}+21{t}^{15}+155{t}^{18}+340{t}^{21}+1249{t}^{24}+2749{t}^{27}+7338{t}^{30}+15172{t}^{33}\\
  &+33561{t}^{36}+63846{t}^{39}+124171{t}^{42}+219049{t}^{45}+386883{t}^{48}+637632{t}^{51}+1043255{t}^{54}\\
  &+1620343{t}^{57}+2488797{t}^{60}+3668275{t}^{63}+5339363{t}^{66}+7513136{t}^{69}+10436320{t}^{72}\\
  &+14086698{t}^{75}+18773622{t}^{78}+24404824{t}^{81}+31336310{t}^{84}+39358417{t}^{87}+48852312{t}^{90}\\
  &+59442770{t}^{93}+71512920{t}^{96}+84482733{t}^{99}+98724670{t}^{102}+113437995{t}^{105}\\
  &+128987065{t}^{108}+144362217{t}^{111}+159945320{t}^{114}+174563826{t}^{117}+188657091{t}^{120}\\
  &+200960031{t}^{123}+212020799{t}^{126}+220567457{t}^{129}+227300866{t}^{132}+231023737{t}^{135}\\
  &+232616715{t}^{138}+231023737{t}^{141}+227300866{t}^{144}+220567457{t}^{147}+212020799{t}^{150}\\
  &+200960031{t}^{153}+188657091{t}^{156}+174563826{t}^{159}+159945320{t}^{162}+144362217{t}^{165}\\
  &+128987065{t}^{168}+113437995{t}^{171}+98724670{t}^{174}+84482733{t}^{177}+71512920{t}^{180}\\
  &+59442770{t}^{183}+48852312{t}^{186}+39358417{t}^{189}+31336310{t}^{192}+24404824{t}^{195}\\
  &+18773622{t}^{198}+14086698{t}^{201}+10436320{t}^{204}+7513136{t}^{207}+5339363{t}^{210}+3668275{t}^{213}\\
  &+2488797{t}^{216}+1620343{t}^{219}+1043255{t}^{222}+637632{t}^{225}+386883{t}^{228}+219049{t}^{231}\\
  &+124171{t}^{234}+63846{t}^{237}+33561{t}^{240}+15172{t}^{243}+7338{t}^{246}+2749{t}^{249}+1249{t}^{252}\\
  &+340{t}^{255}+155{t}^{258}+21{t}^{261}+16{t}^{264}+{t}^{270}+{t}^{276}
\end{align*}
\begin{align*}
  &b_{3,6}= 1+{t}^{5}+{t}^{6}+2{t}^{7}+5{t}^{8}+10{t}^{9}+19{t}^{10}+33{t}^{11}+67{t}^{12}+119{t}^{13}+227{t}^{14}+420{t}^{15}
  +759{t}^{16}\\
  &+1365{t}^{17}+2414{t}^{18}+4173{t}^{19}+7133{t}^{20}+11954{t}^{21}+19723{t}^{22}+32032{t}^{23}+51250{t}^{24}\\
  &+80769{t}^{25}+125568{t}^{26}+192552{t}^{27}+291474{t}^{28}+435769{t}^{29}+643761{t}^{30}+940068{t}^{31}\\
  &+1357776{t}^{32}+1940187{t}^{33}+2744165{t}^{34}+3843089{t}^{35}+5330914{t}^{36}+7326794{t}^{37}\\
  &+9980701{t}^{38}+13479012{t}^{39}+18052285{t}^{40}+23982388{t}^{41}+31611650{t}^{42}+41351888{t}^{43}\\
  &+53695220{t}^{44}+69223771{t}^{45}+88622077{t}^{46}+112687682{t}^{47}+142342980{t}^{48}\\
  &+178646187{t}^{49}+222802513{t}^{50}+276172273{t}^{51}+340280998{t}^{52}+416823728{t}^{53}\\
  &+507669990{t}^{54}+614864095{t}^{55}+740623379{t}^{56}+887330446{t}^{57}+1057524553{t}^{58}\\
  &+1253884087{t}^{59}+1479207593{t}^{60}+1736386956{t}^{61}+2028377410{t}^{62}+2358158471{t}^{63}\\
  &+2728694813{t}^{64}+3142885938{t}^{65}+3603516058{t}^{66}+4113196902{t}^{67}+4674308932{t}^{68}\\
  &+5288936987{t}^{69}+5958810043{t}^{70}+6685233241{t}^{71}+7469027927{t}^{72}+8310469049{t}^{73}\\
  &+9209229814{t}^{74}+10164327764{t}^{75}+11174083807{t}^{76}+12236081611{t}^{77}+13347144073{t}^{78}\\
  &+14503315410{t}^{79}+15699858781{t}^{80}+16931262417{t}^{81}+18191266475{t}^{82}+19472893975{t}^{83}\\
  &+20768505350{t}^{84}+22069859901{t}^{85}+23368194676{t}^{86}+24654311835{t}^{87}+25918683935{t}^{88}\\
  &+27151559407{t}^{89}+28343086653{t}^{90}+29483436265{t}^{91}+30562932672{t}^{92}+31572182881{t}^{93}\\
  &+32502211770{t}^{94}+33344584191{t}^{95}+34091532699{t}^{96}+34736068659{t}^{97}+35272088764{t}^{98}\\
  &+35694464859{t}^{99}+35999126532{t}^{100}+36183117766{t}^{101}+36244648368{t}^{102}+36183117766{t}^{103}\\
  &+35999126532{t}^{104}+35694464859{t}^{105}+35272088764{t}^{106}+34736068659{t}^{107}+34091532699{t}^{108}\\
  &+33344584191{t}^{109}+32502211770{t}^{110}+31572182881{t}^{111}+30562932672{t}^{112}+29483436265{t}^{113}\\
  &+28343086653{t}^{114}+27151559407{t}^{115}+25918683935{t}^{116}+24654311835{t}^{117}+23368194676{t}^{118}\\
  &+22069859901{t}^{119}+20768505350{t}^{120}+19472893975{t}^{121}+18191266475{t}^{122}+16931262417{t}^{123}\\
  &+15699858781{t}^{124}+14503315410{t}^{125}+13347144073{t}^{126}+12236081611{t}^{127}+11174083807{t}^{128}\\
  &+10164327764{t}^{129}+9209229814{t}^{130}+8310469049{t}^{131}+7469027927{t}^{132}+6685233241{t}^{133}\\
  &+5958810043{t}^{134}+5288936987{t}^{135}+4674308932{t}^{136}+4113196902{t}^{137}+3603516058{t}^{138}\\
  &+3142885938{t}^{139}+2728694813{t}^{140}+2358158471{t}^{141}+2028377410{t}^{142}+1736386956{t}^{143}\\
  &+1479207593{t}^{144}+1253884087{t}^{145}+1057524553{t}^{146}+887330446{t}^{147}+740623379{t}^{148}\\
  &+614864095{t}^{149}+507669990{t}^{150}+416823728{t}^{151}+340280998{t}^{152}+276172273{t}^{153}\\
  &+222802513{t}^{154}+178646187{t}^{155}+142342980{t}^{156}+112687682{t}^{157}+88622077{t}^{158}+69223771{t}^{159}\\
  &+53695220{t}^{160}+41351888{t}^{161}+31611650{t}^{162}+23982388{t}^{163}+18052285{t}^{164}+13479012{t}^{165}\\
  &+9980701{t}^{166}+7326794{t}^{167}+5330914{t}^{168}+3843089{t}^{169}+2744165{t}^{170}+1940187{t}^{171}\\
  &+1357776{t}^{172}+940068{t}^{173}+643761{t}^{174}+435769{t}^{175}+291474{t}^{176}+192552{t}^{177}+125568{t}^{178}\\
  &+80769{t}^{179}+51250{t}^{180}+32032{t}^{181}+19723{t}^{182}+11954{t}^{183}+7133{t}^{184}+4173{t}^{185}\\
  &+2414{t}^{186}+1365{t}^{187}+759{t}^{188}+420{t}^{189}+227{t}^{190}+119{t}^{191}+67{t}^{192}+33{t}^{193}\\
  &+19{t}^{194}+10{t}^{195}+5{t}^{196}+2{t}^{197}+{t}^{198}+{t}^{199}+{t}^{204}
\end{align*}


\end{document}